\documentclass[10pt,bezier]{article}
 \usepackage{relsize}
\usepackage{amsmath,amssymb,amsfonts,euscript,graphicx,tikz,tikz-3dplot,amssymb}

\usepackage{tikz,amsmath,tkz-berge}
\usepackage{caption}
\usetikzlibrary{positioning,arrows,calc,decorations.pathmorphing}
\usetikzlibrary{decorations.markings}
\usetikzlibrary{decorations.pathmorphing}

\tikzstyle{vertex}=[circle, fill, draw, inner sep=0pt, minimum size=6pt]

\tdplotsetmaincoords{70}{110}

\newtheorem{prethm}{{\bf Theorem}}

\newenvironment{thm}{\begin{prethm}{\hspace{-0.5
               em}{\bf.}}}{\end{prethm}}

\newtheorem{prepro}[prethm]{{\bf Theorem}}

\newtheorem{preprop}[prethm]{{\bf Proposition}}

\newtheorem{precor}[prethm]{{\bf Corollary}}

\newenvironment{cor}{\begin{precor}{\hspace{-0.5
               em}{\bf.}}}{\end{precor}}

\newtheorem{preconj}[prethm]{{\bf Conjecture}}

\newtheorem{preremark}[prethm]{{\bf Remark}} 

\newenvironment{remark}{\begin{preremark}\rm{\hspace{-0.5
               em}{\bf.}}}{\end{preremark}}

\newtheorem{preexample}[prethm]{{\bf Example}}

\newtheorem{prelem}[prethm]{{\bf Lemma}}

\newenvironment{lem}{\begin{prelem}{\hspace{-0.5
               em}{\bf.}}}{\end{prelem}}

\newtheorem{prelam}{{\bf Lemma}}

\newtheorem{preproof}{{\bf Proof.}}

\newenvironment{proof}[1]{\begin{preproof}{\rm
               #1}\hfill{$\Box$}}{\end{preproof}}

\title{\bf \large  Algebraic Flow Theory of Infinite Graphs
\thanks
{{\it Key Words}:  Contraction, Flow, Infinite graph.}
\thanks {2010{ \it Mathematics Subject Classification}: 05C21, 05C63, 22A05.
 }}

\author{{\normalsize {\sc~${}^{\mathsf{}}$}  {\sc B. Miraftab${}^{\mathsf{a}}$}, {\sc M.J. Moghadamzadeh${}^{\mathsf{b}}$}} \\
 {\footnotesize{${}^{\mathsf{}}$}}\\
 {\footnotesize{${}^{\mathsf{a}}$\it Fachbereich Mathematik, Universit$\ddot{a}$t Hamburg, Bundesstra$\ss$e~$55$,~$20146$ Hamburg, Germany}}\\
{\footnotesize{${}^{\mathsf{b}}$\it Department of Mathematical
 Sciences, Sharif
University of Technology, Tehran, Iran}}\\
{\footnotesize{${}^{\mathsf{}}$ \it {\rm P.O. Box 64615-334},}}\\
{\footnotesize{}}\\
{\footnotesize{
$\mathsf{}$\quad\quad$\mathsf{babak.miraftab@uni\textendash hamburg.de}$}}\\
{\footnotesize{\quad\quad$\mathsf{javad\_mz123@yahoo.com}$}}}

\date{}
\begin{document}
\maketitle
\begin{abstract}
{\small A problem by Diestel is to  extend algebraic flow theory of finite graphs to infinite graphs with ends. In order to pursue this problem, we define an~$A$-flow and non-elusive~$H$-flow for arbitrary graphs and for  abelian Hausdorff topological groups ~$H$ and compact subsets~$A\subseteq H$. We use these new definitions to extend several well-known theorems  of flows in finite graphs to infinite graphs.}
\end{abstract}

\section{Introduction}
{The concept of flow  is a main  topic in graph theory and has various applications, as e.g.\ in  electric networks. Algebraic flow theory for finite graphs is well studied, see \cite{Jaeger1, Jaeger2, lov, sey, thom}.
But when it comes to infinite graphs, much less is known. 
There are some results for electrical networks, see \cite{Aharoni, diestel, sur, Geo}, but not for group-valued flows. 
In fact Diestel's problem \cite[Problem 4.27]{sur} to extend flow theory to infinite graphs is still widely open. 
Here we are doing a first step towards its solution.

\noindent In Section 2, we give our main definition for flows in infinite graphs. 
Roughly speaking, a flow is a map from the edge set of a graph to an abelian  Hausdorff topological group  such that the sum over all edges in each finite cut is trivial. 
With this in mind,  we shall extend the following theorems of finite graphs:
\begin{itemize}
\item A finite graph has a non-elusive~$\mathbb Z_2$-flow if and only if its degrees are even.

\item A finite cubic graph has a non-elusive~$\mathbb Z_4$-flow if and only if it is~$3$-edge-colorable.

\item  Every finite graph containing a Hamilton cycle has a non-elusive~$\mathbb Z_4$-flow.
\end{itemize}
\noindent Our main tool to prove these results is Theorem \ref{compact}, which offers some kind of compactness method to extend results for finite graphs to infinite graphs of arbitrary degree, i.e.\ that need not be locally finite.
\noindent However it is worth remarking that not all theorems about flows in finite graphs have a straightforward analogue in the infinite case:
a  finite cubic graph~$G$ has a non-elusive~$\mathbb Z_3$-flow if and only if~$G$ is bipartite, see \cite[Proposition 6.4.2]{diestel}. 
This is wrong for infinite graphs. Figure 1.1 shows a cubic bipartite graph without any non-elusive~$\mathbb Z_3$-flow.
Even further restrictions on the ends of that graph, e.g.\ requiring them to have edge- or vertex-degree~$3$, fails in our example. 
(For more about the ends of a graph and the topological setting, we refer readers to \cite{sur} and the references therein.)
}

\begin{center}
\begin{tikzpicture}
\draw[<->,>=latex'] (-0.5,1)--(8.5,1);
\draw[<->,>=latex'] (0,0)--(9,0);
\draw[<->,>=latex'] (0.5,-1)--(9.5,-1);

\draw (10,0) node[circle, draw, fill,inner sep=1pt](vv){}node[above=1pt]{$\omega_1$};
\draw (-1,0) node[circle, draw, fill,inner sep=1pt](vv){}node[above=1pt]{$\omega_2$};

\foreach \i in {1,2,...,8}{
\draw (\i,1) node [circle,fill, inner sep=2pt] {};
\draw (\i,0) node [circle,fill, inner sep=2pt] {};
\draw (\i,-1) node [circle,fill, inner sep=2pt] {};
\draw (0,1) node [circle,fill, inner sep=2pt] {};
\draw (9,-1) node [circle,fill, inner sep=2pt] {};

\draw (1,1)--(1,0) ;
\draw (3,1)--(3,0);
\draw (5,1)--(5,0);
\draw (7,1)--(7,0);
\draw (2,0)--(2,-1);
\draw (4,0)--(4,-1);
\draw (6,0)--(6,-1);
\draw (8,0)--(8,-1);
}
\foreach \i in {0,2,4,6,8}{
\draw (\i,1) edge[out=180,in=90] (\i+1,-1);
}
\begin{scope}[shift={(0.5,0.5)}]
\end{scope}
\end{tikzpicture}
\end{center}
\begin{center}
Figure 1.1
\end{center}
In the Section~$4$, we define the concept of tension for infinite graphs. 
Heuristically, a tension is a map from the edge set of a graph to an abelian Hausdorff topological group such that the sum over all edges in each finite cycle is trivial.

\section{Preliminaries}

We refer readers to \cite{diestel}, for the standard terminology and notations in this paper. 
A~$1$-way infinite path is called a \textit{ray}, a~$2$-way infinite path is a \textit{double ray}, and the subrays of a ray or double ray are its tails. Two rays in a graph~$G = (V,E)$ are equivalent if no finite set of vertices separates them. 
This is an equivalence relation whose classes are the \textit{ends} of~$G$. 
Now, consider a locally finite graph~$G$ as one-dimensional CW complex and  compactify~$G$ by using the Freudenthal compactification method. 
We denote this new topological space by~$|G|$, for more on~$|G|$, see \cite{diestel1} and \cite{sur}. 
Let $D$ be a subset of edges of $G$. 
We denote the closure of the point set $\cup_{d\in D} d$ in $|G|$ by $\overline{D}$.
A \textit{circle} in~$|G|$ is a homeomorphic image of the unit circle~$S^1$.
Analogously an \textit{arc} in~$|G|$ is a homeomorphic image of the closed interval $[0,1]$.
We denote the  \textit{cut space}, \textit{finite cut space}, \textit{topological cycle space} and \textit{finite cycle space} of a graph~$G$ by~$\mathcal{B}(G)$,~$\mathcal{B}_{\rm{fin}}(G)$,~$\mathcal{C}(G)$ and~$\mathcal{C}_{\rm{fin}}(G)$, respectively. 
For more details about the equivalent definitions of topological cycle space and its properties, see \cite{diestel,sur}. Note that~$\mathcal{B}(G)$ is a vector space over~$\mathbb Z_2$. 
We now define the degree of an end of the graph $G$.
The~$\textit{edge-degree}$ of an end $\omega$ is the maximum number of edge-disjoint  rays in $\omega$.
In addition, let $D$ be a subset of the set of edges of $G$.
Then we say that an end $\omega$ is \textit{$D$-even} if there exists a finite vertex set $S$ so
that for all finite vertex sets $S'\supseteq S$ it holds that the maximal number of  edge-disjoint  arcs from $S'$ to $\omega$ contained in $\overline{D}$ is even. 
If $D$ is all the edges of $G$, we remove $D$  from the notation and we only say that $\omega$ has an even edge-degree. 
For more about the degree of ends, see \cite{Bruhn, Bruhn2}.
\noindent The following theorem describes the elements of the cycle space for 
locally finite graphs.
For the proof, see {\rm{\cite[Theorem 8.5.10]{diestel}}} and {\rm{\cite[Theorem 5]{Berger}}}.
\begin{thm}\label{cycle space} 
Let~$G=(V,E)$ be a locally finite connected graph.
Then an edge set $D\subseteq E$ lies in $\mathcal{C}(G)$ if and only if one of the following equivalent statements holds
 \begin{enumerate}
\item[{\rm (i)}] $D$ meets every finite cut in an even number of edges.
\item[{\rm (ii)}] Every vertex and every end of $G$ is $D$-even.
 \end{enumerate}
 \end{thm}
 
\noindent Let us review some notions of the compactness method for locally finite graphs. Suppose that~$v_0,v_1,\ldots$ is an enumeration of $V$.
We define~$S_n = {v_0,\ldots,v_n}$, for every~$n\in\mathbb N$.
Put $G_n$ for the minor of $G$ obtained by contracting  each component of $G\setminus S_n$ to a vertex.
Note that we delete any loop, but we keep multiple edges. The vertices of $G_n$ outside $S_n$ are called
\textit{dummy vertices} of $G_n$.
\noindent Let~$G=(V,E)$ be a graph. A \textit{directed edge} is an ordered triple~$(e,x,y)$, where~$e=xy\in E$. So we can present each edge according to its direction by~$\overrightarrow{e}=(e,x,y)$ or~$\overleftarrow{e}=(e,y,x)$. We use~$\overrightarrow{E}$ for the set of all oriented edges of~$G$. 
For two subsets $X,Y$ (not necessarily disjoint) of $V$ and a subset $\overrightarrow{C}$ of $\overrightarrow{E}$, we define 
$$\overrightarrow{C}(X,Y):=\{(e,x,y)\in\overrightarrow{C}\mid x\in X, y\in Y,x\neq y  \}.$$
It is worth mentioning that we can express every finite cut of our graph by a pair $(X,Y)$, where $X$ and $Y=V\setminus X$ are two subsets of the vertices.
Thus for every finite cut $(X,Y)$, we have an oriented cut $\overrightarrow{E}(X,Y)$.
The set $\overrightarrow{\mathcal{B}_{\rm{fin}}}(G)$ denotes the set of all oriented finite cuts  i.e. $\overrightarrow{\mathcal{B}_{\rm{fin}}}(G)=\{\overrightarrow{E}(A,B)\mid (A,B)\in \mathcal{B}_{\rm{fin}}(G)\}$.
Let $H$ be an abelian group(not necessarily finite).
Then we denote all maps from~$\overrightarrow{E}$ to~$H$ such that $f(\overrightarrow{e})=-f(\overleftarrow{e})$ for every non-loop $\overrightarrow{e}\in \overrightarrow{E}$   by~$H^{\overrightarrow{E}}$ and we introduce the following notation only for $\overrightarrow{E}(A,B) \in \overrightarrow{\mathcal{B}_{\rm{fin}}}(G)$ $$f(A,B):=\sum_{\overrightarrow{e}\in \overrightarrow{E}(A,B)} f(\overrightarrow{e}).$$
Also $H^{\overrightarrow{\mathcal{B}_{\rm{fin}}}(G)}$ denotes all maps from ${\mathcal{B}_{\rm{fin}}(G)}$ to $H$ such that $f(A,B)=-f(B,A)$ for every $\overrightarrow{E}(A,B)\in \overrightarrow{\mathcal{B}_{\rm{fin}}}(G)$.
Let us review the definition of group-valued flows for finite graphs\footnote{Our approach is due to \cite{diestel}.}.
A nowhere-zero $H$-flow of the graph $G$ is a map $f\in H^{\overrightarrow{E}}$ with the following properties:
\begin{itemize}
\item[C1:] $f(\overrightarrow{e})\neq 0$, for every $\overrightarrow{e}\in\overrightarrow{E}$.
\item[C2:] $f(\{v\},V)=0$ for all vertices $v$ of $V$.\footnote{This condition is known as the Kirchhoff's law.}
\end{itemize}
A drawback of the above definition is that it depends on degrees of vertices.
So it is meaningless whenever our graph has a vertex with  infinite degree.
To concoct this definition, we switch  every vertex with every oriented cut of our graph in the condition C2 which means $f(A,B)=0$ for all finite cuts $(A,B)$.
More precisely we have the following definition:\\

\noindent {\bf{Definition 1:}} Let~$H$ be an abelian Hausdorff topological group and let~$A$ be a compact subset of~$ H$. 
We define~$\sigma \colon H^{\overrightarrow{E}}\to H^{\overrightarrow{\mathcal{B}_{\rm{fin}}}(G)}$ such that~$$f(X,Y)=\sum_{\overrightarrow{e}\in \overrightarrow{E}(X,Y)} f(\overrightarrow{e})$$ for any finite oriented cut~$\overrightarrow{E}(X,Y)$. 
Let~$M$ be a subset of~$\overrightarrow{{\mathcal{B}_{\rm{fin}}}}(G)$. 
Then we say that~$G$ has an \textit{$A$-flow} with respect to~$M$  if~$F_M=\{f\in A^{\overrightarrow{E}}\,|\,\sigma(f)(\overrightarrow{E}(X,Y))=0$ for every~$\overrightarrow{E}(X,Y)\in M\}$ is not empty and we say that~$G$ has an \textit{$A$-flow} if~$G$ has an \textit{$A$-flow} with respect to~$\overrightarrow{\mathcal{{B}_{\rm{fin}}}}(G)$. 
If $f$ is an $A$-flow and $A\subseteq H\setminus \{0\}$, then we also call $f$ a \emph{non-elusive $H$-flow}.

\noindent {\bf{Definition 2:}} With the above notation, suppose that $G$ has an $A$-flow, where  $H=\mathbb Z$ with the discrete topology and $A=\{-(k-1),\ldots,k-1\}\setminus\{0\}$. 
Then we say that~$G$ has a \textit{$k$-flow}.
 
\noindent If a graph~$G$ has more than one component, then~$G$ has an~$A$-flow if and only if each of its components does. That is why we restrict ourselves to connected graphs from now on. So let~$G$ be a connected graph.\\
\noindent It is worth mentioning that if~$G$ is locally finite, then this definition coincides with the one in Section~$4.3$ of \cite{sur} for abelian groups. If the graph~$G$ is locally finite, then using the compactness method, we can  generalize almost all theorems of finite flow theory to infinite. 

\noindent {\bf{Definition 3:}} Let~$M=\{C_{1},\ldots,C_{t}\}$ be a finite subset of $\mathcal{B}_{\rm{fin}}(G)$.
Then we define a multigraph~$G_M$ according to~$M$. Each cut~$C_{i}\in M$ belongs to a bipartition~$(A_i,B_i)$ of~$V$  such that~$C_i$ are the~$A_i-B_i$ edges. 
The vertices of~$G_M$ are the words $X_{1}\cdots X_{t}$, where~$X_i\in\{A_i,B_i\}$ for~$i=1,\ldots,t$ in such a way that~$\cap_{i=1}^{t} X_{i}\neq\varnothing$. 
Between two vertices~$X_{1}\cdots X_{t}$ and~$X^{\prime}_{1}\cdots X^{\prime}_{t}$ of~$G_M$, there is an edge for each edge  between~$\bigcap_{i=1}^t X_{i}$ and~$\bigcap_{i=1}^t X'_{i}$. 
We say that~$G_M$ is obtained from~$G$ by \emph{contracting} with respect to~$M$. 

\begin{remark}\label{withoutloops}
Let $G$ be an infinite graph and let $M$ be a finite subset of $\mathcal{B}_{\rm{fin}}(G)$.
Then throughout this paper we always first consider $G_M$ without its loops and then we apply the corresponding result for finite multigraphs.
Now we  extend the flow in an arbitrary way to the loops.
This is possible, as no cut of $G$ contains  a loop and so the assignments of loops do not influence whether our function is a flow or not.
\end{remark}

\begin{remark}\label{phi} The definition of~$G_M$ leads to a map~$\phi\colon G\rightarrow G_M$, where every vertex~$u$ of~$G$ is mapped to a unique word~$V_u\in V(G_M)$, it is contained in. 
Indeed,  looking at each finite cut in~$M$, we can construct the unique word~$X_{_1}\cdots X_{t}$ 
in such a way that every~$X_{i}$ contains~$u$, for each~$i\in\{1,\ldots,t\}$ and so~$u\in \bigcap_{i=1}^t X_i$. 
We notice that each edge of~$G$ induces an edge of~$G_M$. Indeed, it is not hard to see that~$\phi$ defines a bijective map on the set of edges. Also, it is worth mentioning that~$\phi^{-1}(U_1)\cap\phi^{-1}(U_2)=\varnothing$ for every two vertices~$U_1$ and~$U_2$ of~$V(G_M)$. Thus the vertex set of~$G_M$ is a partition of~$V$. 
\end{remark}
Our compactness method is more general than the ordinary compactness method for locally finite graphs as mentioned above. When the graph~$G$ is locally finite, for each~$G_n$\footnote{For definition of $G_n$, see Preliminaries.}, we can choose a suitable subset~$M$ of the set of finite cuts such that~$G_M$ coincides with~$G_n$.


\section{Flows on Infinite Graphs}

\noindent  First, we start with the following lemma.
\begin{lem}\label{cut}
  Let~$G$ be a graph and~$M$ be a finite subset of~$\mathcal{B}_{\rm{fin}}(G)$. 
  Then we have~$M\subseteq\mathcal{B}(G_M)\subseteq \mathcal{B}_{\rm{fin}}(G)$.
\end{lem}
\begin{proof}{
First, we show that~$M\subseteq\mathcal{B}(G_M)$. Let~$C=E(A,B)\in M$. 
Then consider the set of all words containing~$A$ and do the same for all words containing~$B$, say~$\mathcal{A}$ and~$\mathcal{B}$, respectively. 
The sets~$\mathcal{A}$ and~$\mathcal{B}$ form a partition of~$G_M$ and so  we have~$C$ as a cut of~$G_M$. 
Note that~$\mathcal{A}$ and~$\mathcal{B}$ are not empty, since every~$uv\in C$ induces vertices~$V_u\in \mathcal{A}$ and~$V_v\in \mathcal{B}$. 
Now, assume that~$C=E(A,B)\in\mathcal{B}(G_M)$. 
We deduce from Remark \ref{phi} that the edges between~$A$ and~$B$  in~$G_M$ are those between~$\phi^{-1}(A)$ and~$\phi^{-1}(B)$. 
Hence~$(\phi^{-1}(A),\phi^{-1}(B))$ forms a partition of~$G$ and so~$C$ is a finite cut of~$G$.
}\end{proof}
\noindent The following theorem plays a vital role in this paper and is a basic key to generalize flow theory of finite to infinite graphs.
\begin{thm}\label{compact}
Let~$G$ be a graph and~$H$ be an abelian Hausdorff topological group with compact subset~$A$. 
Then~$G$ has an~$A$-flow if and only if~$G_M$ has an~$A$-flow for every finite subset~$M$ of~$\mathcal{B}_{\rm{fin}}(G)$.
\end{thm}
\begin{proof}{
First, assume that~$G$ has an~$A$-flow. 
By Lemma \ref{cut}, every finite cut of~$G_M$ belongs to~$\mathcal{B}_{\rm{fin}}(G)$. 
So every~$A$-flow of~$G$ is an~$A$-flow of~$G_M$. 
In particular,~$G_M$ has some~$A$-flow. 
For the backward implication, since~$H$ is a topological group, the sets~$H^{\overrightarrow{E}}$ and 
$H^{\overrightarrow{\mathcal{B}_{\rm{fin}}}(G)}$ are endowed with the product topology.  
Let~$M=\{C_{1},\ldots,C_{t}\}$ be a subset of~${\overrightarrow{\mathcal{B}_{\rm{fin}}}(G)}$. 
We define~$\sigma_{i}\colon H^{\overrightarrow{E}}\to H$ by~$\sigma_{i}(f)=\sum_{e\in C_{i}} f(e)$.
Since the sum operation is a continuous map,~$\sigma_{i}$ is continuous for each~$i$. 
Therefore~$\sigma_{i}^{-1}(0)$ is a closed subspace in~$H^{\overrightarrow{E}}$, as~$H$ is Hausdorff. 
On the other hand, by Tychonoff's theorem (see \cite[Theorem 37.3]{munkres}),~$A^{\overrightarrow{E}}$ is compact and so is~$\sigma_{ i}^{-1}(0)\cap A^{\overrightarrow{E}}$. 
It is clear that~$F_M=\bigcap_{i=1}^t\sigma_{i}^{-1}(0)\cap A^{\overrightarrow{E}}$ and so~$F_M$ is compact. 
Since $G_M$ has an $A$-flow, by definition, the set $F_{{\mathcal B}(G_M)}$ is not empty. 
Lemma \ref{cut} implies that $F_M$ is not empty.
Hence the intersection of every finite family of~$F_{\{C_{i}\}}$ with~$C_{i}\in\mathcal{B}_{\rm{fin}}(G)$ is 
not empty.
Since~$A^{\overrightarrow{E}}$ is compact, we deduce that~$F_{\mathcal{B}_{\rm{fin}}(G)}=\underset{C_{i}\in\mathcal{B}_{\rm{fin}}(G)} {\bigcap}F_{\{C_{i}\}}$ is not empty, see \cite[Theorem 26.9]{munkres}. 
Thus~$G$ has an~$A$-flow.
}\end{proof}
For finite graphs, the existence of a nowhere-zero~$H$-flow does not depend on the structure of~$H$ but only on its order, see \cite[Corollary 6.3.2]{diestel}. 
In the next corollary, we show that the same is true for infinite graphs.   
\begin{cor}\label{HH'}
Let~$H$ and~$H'$ be two finite abelian groups with equal order. 
Then~$G$ has a non-elusive~$H$-flow if and only if~$G$ has a non-elusive~$H'$-flow. 
\end{cor}
\begin{proof}{
We note that~$H$ and~$H'$  are endowed by the discrete topologies and so they are compact. 
Suppose~$G$ has a non-elusive~$H$-flow.
By Theorem \ref{compact}, for every finite subset~$M$ of~$\mathcal{B}_{\rm{fin}}(G)$, the multigraph~$G_M$ has a   non-elusive~$H$-flow. 
We notice that $G_M$ might have infinitely many loops.
Since each loop appears twice, we can ignore them and so we only care the rest of edges which are finite.
Thus we are able to apply  \cite[Corollary 6.3.2]{diestel} and conclude that every~$G_M$ has  a non-elusive~$H'$-flow. 
Again, it follows from Thereom \ref{compact} that~$G$ has a  non-elusive~$H$-flow. 
The other direction follows from the symmetry of the statement.~}\end{proof}
  
  \noindent There is a direct connection between~$k$-flows and non-elusive~$\mathbb Z_k$-flows in finite graphs which was discovered by Tutte, see \cite{tutte}. In the next corollary, we use Theorem \ref{compact} and show that having a~$k$-flow and a non-elusive~$\mathbb Z_k$-flow
   are equivalent in infinite graphs. 

\begin{cor}\label{kZk}
  A graph admits a~$k$-flow if and only if it admits a non-elusive~$\mathbb Z_k$-flow.
\end{cor}
\begin{proof}
{ The canonical homomorphism~$\mathbb Z\rightarrow{\mathbb Z_k}$ implies the forward 
 implication. For the converse, assume that~$G$ has a non-elusive~$\mathbb Z_k$-flow.  By Theorem \ref{compact}, for every finite subset~$M$ of~$\mathcal{B}_{\rm{fin}}(G)$, the multigraph~$G_M$ has  a non-elusive~$\mathbb Z_k$-flow.
 We consider~$\mathbb Z_k$  with the discrete topology. It follows from Theorem \ref{compact} and \cite{tutte} that every~$G_M$ has 
  a~$k$-flow. Again, we invoke Theorem \ref{compact} and we conclude that~$G$ has a~$k$-flow.}
\end{proof}

\noindent Next up, we study non-elusive~$\mathbb Z_m$-flows for some special values of~$m$. First, we study non-elusive~$\mathbb Z_2$-flows
 for locally finite graphs. It is worth mentioning that if~$G$ is an arbitrary infinite graph and~$G$ has a 
 non-elusive~$\mathbb Z_2$-flow, then one can see that all finite cuts of~$G$ are even and vice versa.
First we need a notation.
Suppose that $G=(V,E)$ is a graph and $F$ is a subset of $E$.
We define the indicate function $\delta_F\colon E\to \mathbb Z_2$ in the following way:\\
$$\delta_F(e):=\left\{ \begin{array}{rcl}
         1 & \mbox{for}
         & e\in F \\ 0  & \mbox{for} & e\notin F 
                \end{array}\right.
$$

\begin{thm}\label{Z_2}
Let~$G=(V,E)$ be a locally finite graph and let $F$ be a subset of $E$.
 Then $\delta_F$ is a~$\mathbb Z_2$-flow if and only $F\in \mathcal{C}(G)$.
\end{thm} 
\begin{proof}{
First suppose that $\delta_F$ is a~$\mathbb Z_2$-flow.
It is not hard to see that every vertex and every end of $G$ is $F$-even.
So it follows from Theorem \ref{cycle space} that $F$ belongs to the cycle space of $G$.
For the backward implication, since $F\in\mathcal{C}(G)$, we are able to invoke Theorem \ref{cycle space} and conclude that every vertex and every end of $G$ is $F$-even.
Thus it implies that $\delta_F$ is a $\mathbb Z_2$-flow.
}\end{proof}
\noindent It is not hard to see that if a cubic graph~$G$ has a non-elusive~$\mathbb Z_3$-flow, then~$G$ is bipartite.
For a cubic graph $G$, having a non-elusive $\mathbb Z_3$-flow is equivalent to having an orientation of $G$ in such a way that for every vertex $v$ of $G$ all incident edges of $v$ are either directed outward or directed inward and moreover all assignments are one. 
Let $G$ be a graph as depicted on Figure 1.1. 
Consider orientations with the above property.
So we have two cases.
In each case, we have a finite cut whose sum of assignments is not zero, see Figure 3.0.1.
\begin{center}
\begin{tikzpicture}
[x=1.3cm, y=1cm,
    every edge/.style={
        draw,
        postaction={decorate,
                    decoration={markings,mark=at position 0.5 with {\arrow[scale=0.09cm]{>}}}
                   }
        }
]

\draw (0,1) node [circle,fill, inner sep=2pt] {};
\draw (2,1) node [circle,fill, inner sep=2pt] {};
\draw (0,0) node [circle,fill, inner sep=2pt] {};
\draw (2,0) node [circle,fill, inner sep=2pt] {};
\draw (0,-1) node [circle,fill, inner sep=2pt] {};
\draw (2,-1) node [circle,fill, inner sep=2pt] {};

\draw (9,1) node [circle,fill, inner sep=2pt] {};
\draw (7,1) node [circle,fill, inner sep=2pt] {};
\draw (9,0) node [circle,fill, inner sep=2pt] {};
\draw (7,0) node [circle,fill, inner sep=2pt] {};
\draw (9,-1) node [circle,fill, inner sep=2pt] {};
\draw (7,-1) node [circle,fill, inner sep=2pt] {};

\path
		(9,-1) edge (7,-1)
		(9,1) edge (7,1)
		(0,-1) edge (2,-1)
		(7,0) edge (9,0)
	    (0,1)edge(2,1)
		(2,0)  edge (0,0) 
		(0,-1) edge (2,-1) 
        
		;
\draw[dashed] (2,-1.4) arc (90:270:0.6cm and -1.36cm);
\draw[dashed] (0.1,1.3) arc (-90:90:0.6cm and -1.36cm);
\draw[dashed] (9,-1.4) arc (90:270:0.6cm and -1.36cm);
\draw[dashed] (7,1.3) arc (-90:90:0.6cm and -1.36cm);
\begin{scope}[shift={(0.5,0.5)}]
\end{scope}
\end{tikzpicture}
\end{center}
\begin{center}
Figure 3.0.1
\end{center}

\noindent Hence, we propose this question: When does  a cubic graph has a non-elusive~$\mathbb Z_3$-flow?\\
Recently, Thomassen used~$S^1$ and~$R_3$ in flow theory of finite multigraphs and investigated the connection of such flows with~$\mathbb Z_3$-flows for finite multigraphs, see \cite{thom2}.
Now let us review these notations here. 
Let $G=(V,E)$ be a finite multigraph without loops.
Then an~$S^1$-flow is the same as a flow whose flow values are complex numbers with absolute value~$1$.
But we first choose an orientation for each $e\in E$ and then we assign elements of $S^1$ on the edges.\footnote{We follow this approach only for the next three results.} 
Let~$R_k$ denote the set of~$k$-th roots of unity, that is, the solutions to the equation 
$z^k=1$. 
\begin{lem} {\rm\cite[Proposition 1]{thom2}}\label{cubic}
 Let~$G$ be a finite multigraph without loops. Then {\rm (i)} and {\rm (ii)} below are equivalent, and they imply the statement {\rm (iii)}  
\begin{enumerate}
\item[{\rm (i)}]~$G$ has a non-elusive~$\mathbb Z_3$-flow.
\item[{\rm (ii)}]~$G$ has an~$R_3$-flow.
\item[{\rm (iii)}]~$G$ has an~$S^1$-flow.
\end{enumerate}
\noindent If~$G$ is cubic, the three statements are equivalent, and~$G$ satisfies
{\rm (i)}, {\rm (ii)}, {\rm (iii)}  if and only if~$G$ is bipartite.
\end{lem}
We generalize Lemma \ref{cubic}. We replace the condition \emph{cubic} with an edge dominating set~$H$ of vertices such 
that the degree of every vertex of~$H$ is~$3$. 
A subset~$H$ of vertices is an \textit{edge dominating set} if every edge of the graph has an end vertex in~$H$.
\begin{lem}\label{ourcubic}
Let~$G$ be a finite  multigraph without loops with a connected edge dominating set~$U$ of vertices such that every vertex of~$U$ has degree~$3$. 
Then the following three statements are equivalent.
\begin{enumerate}
\item[{\rm (i)}]~$G$ has a non-elusive~$\mathbb Z_3$-flow.
\item[{\rm (ii)}]~$G$ has an~$R_3$-flow.
\item[{\rm (iii)}]~$G$ has an~$S^1$-flow.
\end{enumerate}
\end{lem}
\begin{proof}{
 By Lemma \ref{cubic}, it is enough to show that  {\rm (iii)}~$\Rightarrow$ {\rm (ii)}. 
 One may suppose that $G$ has at least one edge.
 Assume that~$G$ has an~$S^1$-flow, say~$f$. 
 Choose an edge of~$G$, say~$uv$ with~$u\in U$. 
 We notice that $U$ contains at least two vertices.
 Because if $U$ has only one vertex, then every vertex in $V\setminus U$ would have degree one and so we are not able to have an $S^1$-flow. 
 Let~$f(uv)=z_1\in S^1$. 
 Since~$f$ is an~$S^1$-flow,  there are~$z_2,z_3\in S^1$ such that~$z_1+z_2+z_3=0$. 
 Note that~$z_2$ and~$z_3$ are unique. 
 Let~$w$ be a neighbour of~$u$ in~$U$. 
 Then degree of~$w$ is three and so the values of~$f$ on edges incident to~$w$ lie exactly in the set~$\{z_1,z_2,z_3\}$.
 Since~$U$  is connected and meets every edge of~$G$, we know that~$f$ assigns~$z_1,z_2$ or~$z_3$ to every edge of~$G$. 
 Thus~$f$ is a~$\{z_1,z_2,z_3\}$-flow on~$G$.
 Since there is a bijection between~$\{z_1,z_2,z_3\}$ and~$R_3$, we find an~$R_3$-flow for~$G$.~}\end{proof} 
 

\noindent Now, we are ready to answer this question: When does a cubic graph have a non-elusive~$\mathbb 
Z_3$-flow?
\begin{thm}
  If~$G$ is a cubic graph, then the following statements are 
equivalent.
\begin{enumerate}
\item[{\rm (i)}]~$G$ has a non-elusive~$\mathbb Z_3$-flow.
\item[{\rm (ii)}]~$G$ has an~$R_3$-flow.
\item[{\rm (iii)}]~$G$ has an~$S^1$-flow.
\end{enumerate}

\end{thm}
\begin{proof}{
(i)~$\Rightarrow$ (ii)  It follows from Theorems \ref{compact} and \ref{cubic} that for every finite subset~$M$ of~$\mathcal{B}_{\rm{fin}}(G)$, the multigraph~$G_M$ has an~$R_3$-flow. 
So by Theorem \ref{compact},~$G$ has an~$R_3$-flow.  
(ii)~$\Rightarrow$ (iii) is trivial. 
(iii)~$\Rightarrow$ (i) By Theorem \ref{compact}, the multigraph~$G_M$ has an~$S^1$-flow. 
We notice that as we mentioned in Remark \ref{withoutloops}, we ignore all loops of $G_M$.
Let~$U$ be the set of all vertices that are incident with an edge from a cut of $M$. 
We note that~$U$ is finite. 
We add some paths of~$G$ to~$G[U]$ until we get a connected graph~$N$. 
Note that it suffices to take only finitely many paths, i.e.\ we may assume that~$N$ is finite. 
Let~$S_N$ be the set of vertices of~$N$ and assume that~$G_N$ is obtained by contracting the components of~$G\setminus S_N$ to dummy vertices, similar to constructing of multigraph~$G_n$ for the compactness method. Obviously,~$S_N$ is an edge dominating set of vertices of~$G_N$ and moreover the degree of each vertex of~$N$ is~$3$.
We notice that $G_N$ has an $S^1$-flow, as $G$ has an $S^1$-flow.
By Lemma \ref{ourcubic}, the multigraph~$G_N$ has a non-elusive~$\mathbb Z_3$-flow. 
Since every element of~$M$ is a cut of~$G_N$, the multigraph~$G_M$ has a non-elusive~$\mathbb Z_3$-flow. 
We invoke Theorem \ref{compact} and we conclude that~$G$ has a non-elusive~$\mathbb Z_3$-flow.}\end{proof}

\noindent Next, we study non-elusive~$\mathbb Z_4$-flows.
\begin{thm}\label{Z_4} Let~$G=(V,E)$ be a locally finite graph. 
Then~$G$ has a non-elusive~$\mathbb Z_4$-flow if and only if $E$ is the union of two elements of its topological cycle space.\end{thm}
\begin{proof}{ 
First, suppose that~$G$ has a non-elusive~$\mathbb Z_4$-flow. By Corollary \ref{HH'}, we can assume that~$G$ has a non-elusive $\mathbb Z_2\oplus\mathbb Z_2$-flow, say~$f$. Set~$E_i=\{e\in E(G)\mid\pi_i(f(e))\neq 0\}$ 
for~$i=0,1$, where~$\pi_1$ and~$\pi_2$ are the projection maps on the first and second coordinates, respectively. Since~$G$ has a~$\mathbb Z_2\oplus\mathbb Z_2$-flow, each finite cut of $G$ meets $E_i$ evenly.
We now invoke Theorem \ref{cycle space} and  conclude that every~$E_i$ belongs to the topological cycle space of~$G$, for~$i=0,1$. 
For the backward implication, let~$G=G_1\cup G_2$ with~$E(G_i)\in \mathcal C(G)$, for~$i=1,2$. 
It follows from  Theorem \ref{cycle space} and  Theorem \ref{Z_2} that each~$G_i$ has a non-elusive~$\mathbb Z_2$-flow. 
Thus we can find a non-elusive~$\mathbb Z_2\oplus\mathbb Z_2$-flow and by Theorem \ref{HH'}, we are done.~}\end{proof}

\subsection{Edge-coloring for infinite graphs}

If~$G$ is a cubic finite graph, then  the conditions of having a non-elusive~$\mathbb Z_4$-flow and~$3$-edge-colorability of~$G$ are equivalent, but this is not true for infinite graphs. 
Let $G$ be a graph as depicted in Figure 3.1.1.
Suppose that $G$ has a non-elusive $\mathbb Z_4$-flow.
On the other hand, we are able to contract  the graph $G$ to the Petersen graph.
But it is known that the Petersen graph is not $3$-edge-colorable.
So we deduce that the Petersen graph does not admit a non-elusive $\mathbb Z_4$-flow and it implies that $G$ does not have a non-elusive $\mathbb Z_4$-flow, either.

It seems that the notion of edge-coloring is not suitable for a characterization of when an infinite graph with ends admits a~$k$-flow, but that a generalization of edge-colorability(\emph{``semi-k-edge-colorability"}, to be defined below) is. We only need a definition
 of generalized edge-coloring for cubic graphs here which implies the existence of a  non-elusive~$\mathbb Z_4$-flow. Hence we will define this concept under the name of semi-edge-coloring.
Next, we define semi-edge-coloring for~$k$-regular graphs where~$k$ is an odd 
number and we show that this definition for cubic graphs is equivalent to having a non-elusive~$\mathbb Z_4$-flow.
\begin{center}
\begin{tikzpicture}[style=thick]
\begin{scope}[shift={(0.5,0.5)}]
\draw (0,1.5) node [circle,fill, inner sep=1.5pt] {};
\draw (0,0.5) node [circle,fill, inner sep=1.5pt] {};
\draw (1,2.10) node [circle,fill, inner sep=1.5pt] {};
\draw (2,1.5) node [circle,fill, inner sep=1.5pt] {};
\draw (2,0.5) node [circle,fill, inner sep=1.5pt] {};
\draw (0.5,-0.1) node [circle,fill, inner sep=1.5pt] {};
\draw (1.5,-0.1) node [circle,fill, inner sep=1.5pt] {};
\draw (0.5,1) node [circle,fill, inner sep=1.5pt] {};
\draw (1,1.5) node [circle,fill, inner sep=1.5pt] {};
\draw (1.5,1) node [circle,fill, inner sep=1.5pt] {};
\draw (0.69,0.5) node [circle,fill, inner sep=1.5pt] {};
\draw (1.26,0.5) node [circle,fill, inner sep=1.5pt] {};

\draw (0,1.5)--(0,0.5)[ densely dotted];
\draw[dashed] (0,1.5)--(1,2.10)[];
\draw (0,1)--(0.5,1)[];
\draw[dashed] (0,0.5) --(0.5,-0.1)[];
\draw (0.5,-0.1)--(1.5,-0.1)[ densely dotted];
\draw (1.5,-0.1)-- (2,0.5)[];
\draw (2,1.5)--(1,2.10)[];

\draw[dashed] (1,1.5) --(0.69,0.5)[];
\draw (0.69,0.5)-- (1.5,1)[ densely dotted];
\draw (1.5,1)--(0.5,1)[];
\draw (1.26,0.5)--(0.5,1)[ densely dotted];
\draw (1.26,0.5)--(1,1.5)[];

\draw (1,1.5)--(1,2.10)[ densely dotted];
\draw (0.69,0.5)--(0.5,-0.1)[];
\draw[dashed] (1.26,0.5)--(1.5,-0.1)[];

\draw (2,1.5) node [circle,fill, inner sep=1.5pt] {};
\draw (2,0.5) node [circle,fill, inner sep=1.5pt] {};
\draw (3,2.10) node [circle,fill, inner sep=1.5pt] {};
\draw (4,1.5) node [circle,fill, inner sep=1.5pt] {};
\draw (4,0.5) node [circle,fill, inner sep=1.5pt] {};
\draw (2.5,-0.1) node [circle,fill, inner sep=1.5pt] {};
\draw (3.5,-0.1) node [circle,fill, inner sep=1.5pt] {};
\draw (2.5,1) node [circle,fill, inner sep=1.5pt] {};
\draw (3,1.5) node [circle,fill, inner sep=1.5pt] {};
\draw (3.5,1) node [circle,fill, inner sep=1.5pt] {};
\draw (2.69,0.5) node [circle,fill, inner sep=1.5pt] {};
\draw (3.26,0.5) node [circle,fill, inner sep=1.5pt] {};

\draw[dashed] (2,1.5)--(3,2.10)[];
\draw (2,1)--(2.5,1)[];
\draw[dashed] (2,0.5) --(2.5,-0.1)[];
\draw (2.5,-0.1)--(3.5,-0.1)[densely dotted];
\draw (3.5,-0.1)-- (4,0.5)[];
\draw (4,1.5)--(3,2.10)[];

\draw[dashed] (3,1.5) --(2.69,0.5)[];
\draw (2.69,0.5)-- (3.5,1)[ densely dotted];
\draw (3.5,1)--(2.5,1)[];
\draw (3.26,0.5)--(2.5,1)[ densely dotted];
\draw (3.26,0.5)--(3,1.5)[];

\draw (3,1.5)--(3,2.10)[densely dotted];
\draw (2.69,0.5)--(2.5,-0.1)[];
\draw[dashed] (3.26,0.5)--(3.5,-0.1)[];
\draw (2.5,1)--(1.5,1)[];
\draw (3.5,1)--(4.5,1)[];

\draw (4,1.5) node [circle,fill, inner sep=1.5pt] {};
\draw (4,0.5) node [circle,fill, inner sep=1.5pt] {};
\draw (5,2.10) node [circle,fill, inner sep=1.5pt] {};
\draw (6,1.5) node [circle,fill, inner sep=1.5pt] {};
\draw (6,0.5) node [circle,fill, inner sep=1.5pt] {};
\draw (4.5,-0.1) node [circle,fill, inner sep=1.5pt] {};
\draw (5.5,-0.1) node [circle,fill, inner sep=1.5pt] {};
\draw (4.5,1) node [circle,fill, inner sep=1.5pt] {};
\draw (5,1.5) node [circle,fill, inner sep=1.5pt] {};
\draw (5.5,1) node [circle,fill, inner sep=1.5pt] {};
\draw (4.69,0.5) node [circle,fill, inner sep=1.5pt] {};
\draw (5.26,0.5) node [circle,fill, inner sep=1.5pt] {};

\draw[dashed] (4,1.5)--(5,2.10)[];
\draw[dashed] (4,0.5) --(4.5,-0.1)[];
\draw (4.5,-0.1)--(5.5,-0.1)[ densely dotted];
\draw (5.5,-0.1)-- (6,0.5)[];
\draw (6,1.5)--(5,2.10)[];

\draw[dashed] (5,1.5) --(4.69,0.5)[];
\draw (4.69,0.5)-- (5.5,1)[ densely dotted];
\draw[dashed] (5.5,1)--(4.5,1)[];
\draw (5.26,0.5)--(4.5,1)[ densely dotted];
\draw (5.26,0.5)--(5,1.5)[];

\draw (5,1.5)--(5,2.10)[ densely dotted];
\draw (4.69,0.5)--(4.5,-0.1)[];
\draw[dashed] (5.26,0.5)--(5.5,-0.1)[];
\draw (5.5,1)--(6,1)[];
\draw (6,1.5)--(6,0.5)[ densely dotted];
\draw (4,1.5)--(4,0.5)[ densely dotted];
\draw (2,1.5)--(2,0.5)[ densely dotted];
\draw (8,1.5)--(8,0.5)[ densely dotted];

\draw[->,>=latex',dashed] (8,1.5)--(9,1.5)[];
\draw[->,>=latex',dashed] (8,0.5)--(9,0.5)[];
\draw[->,>=latex'] (0,1.5)--(-1,1.5)[];
\draw[->,>=latex'] (0,1)--(-1,1)[];
\draw[->,>=latex'] (0,0.5)--(-1,0.5)[];

\draw (6,1.5) node [circle,fill, inner sep=1.5pt] {};
\draw (6,0.5) node [circle,fill, inner sep=1.5pt] {};
\draw (7,2.10) node [circle,fill, inner sep=1.5pt] {};
\draw (8,1.5) node [circle,fill, inner sep=1.5pt] {};
\draw (8,0.5) node [circle,fill, inner sep=1.5pt] {};
\draw (6.5,-0.1) node [circle,fill, inner sep=1.5pt] {};
\draw (7.5,-0.1) node [circle,fill, inner sep=1.5pt] {};
\draw (6.5,1) node [circle,fill, inner sep=1.5pt] {};
\draw (7,1.5) node [circle,fill, inner sep=1.5pt] {};
\draw (7.5,1) node [circle,fill, inner sep=1.5pt] {};
\draw (6.69,0.5) node [circle,fill, inner sep=1.5pt] {};
\draw (7.26,0.5) node [circle,fill, inner sep=1.5pt] {};

\draw[dashed] (6,1.5)--(7,2.10)[];
\draw (6,1)--(6.5,1)[];
\draw[dashed] (6,0.5) --(6.5,-0.1)[];
\draw (6.5,-0.1)--(7.5,-0.1)[ densely dotted];
\draw (7.5,-0.1)-- (8,0.5)[];
\draw (8,1.5)--(7,2.10)[];

\draw[dashed] (7,1.5) --(6.69,0.5)[];
\draw (6.69,0.5)-- (7.5,1)[ densely dotted];
\draw[dashed] (7.5,1)--(6.5,1)[];
\draw (7.26,0.5)--(6.5,1)[densely dotted];
\draw (7.26,0.5)--(7,1.5)[];

\draw(7,1.5)--(7,2.10)[ densely dotted];
\draw (6.69,0.5)--(6.5,-0.1)[];
\draw[dashed] (7.26,0.5)--(7.5,-0.1)[];
\draw (7.5,1)--(9,1)[->];

\end{scope}
\end{tikzpicture}
\end{center}
   \begin{center}
  Figure 3.1.1
\end{center}
\noindent Before defining this new edge-colorability,  note that we can define~$k$-flow 
axiomatically for finite graphs. 
Our objective is to show that every graph which has a  $k$-flow is a contraction of a cubic graph which has a $k$-flow. 
In order to show this, we need a definition. 
We call a map~$\mathcal F$ from the class of all finite graphs to~$\mathbb Z_2$ a \emph{``Boolean functor of having the property~$P$"} if~$G$ has the property~$P$ if and only if~$\mathcal {F}(G)=1$.   
For instance, having a $k$-flow is a Boolean functor.
We denote it by $\mathcal F$.
We notice that if $\mathcal{F}(G)=1$ for a given graph $G$, then $\mathcal{F}(H)=1$, where $H$ is a contraction of $G$.\\
\noindent The property of admitting a~$k$-flow or equivalently a non-elusive $\mathbb Z_k$-flow can be characterized as follows:
\begin{thm}\label{havingflow}
  Let~$k>2$ be an odd number and~$\mathcal{F}$ be the Boolean functor of having a non-elusive~$\mathbb Z_k$-flow for every finite graph and~$\mathcal{F}'$ be another Boolean functor which satisfy the following three properties. 
\begin{enumerate}
\item[{\rm (i)}]~$\mathcal{F}$ and~$\mathcal{F}'$ are the same for cubic graphs.
\item[{\rm (ii)}] If~$\mathcal{F}'(G)=1$, then~$\mathcal{F}'(H)=1$ for every contraction\footnote{The contracted vertex sets need not be connected.}~$H$ of~$G$.
\item[{\rm (iii)}] If~$\mathcal{F}'(G)=1$, then there is a cubic graph~$H$  with~$\mathcal{F}'(H)=1$ such that~$G$ is a contraction of~$H$.
\end{enumerate} 
Then~$\mathcal{F}$ and~$\mathcal{F}'$ are equal.
\end{thm}
\begin{proof}{
Assume that~$\mathcal{F}(G)=1$, for a given finite graph~$G$ and let $f$ be a non-elusive $\mathbb Z_k$-flow of $G$. 
We now introduce a cubic graph~$H'$ such that~$G$ is a contraction of~$H'$ and $\mathcal{F}'(H')=1$. 
Our strategy is to switch all vertices with degrees at least four with vertices with degrees at most three and then we eliminate all vertices with degrees two. 
Let~$v\in V(G)$ with degree at least four. 
Suppose that the sum of values of two edges~$e_1$ and~$e_2$ that are incident with~$v$ is~$0$. 
First, we add a new vertex~$u$. 
Then we separate these two edges from~$v$ and we join~$e_1$ and~$e_2$ to~$u$. 
In other words, the degree of~$u$ is two and~$e_1$ and~$e_2$ are incident to~$u$. 
So the degree of the vertex~$v$ reduces by~$2$ in the new graph.
Now we assume that there are two edges which are incident to~$v$ and the sum of their flows is not~$0$, say~$e_1$ and~$e_2$.
We separate~$e_1$ and~$e_2$ from~$v$ with a new vertex~$u$ like in the previous case and join the new vertex~$u$ to~$v$. 
In other words, we substitute these two edges with a claw i.e.~$K_{1,3}$. 
We continue this process for all vertices of~$G$ until~$\Delta(G)\leq 3$ is obtained. 
We call the new graph~$H$. 
Next we are going to replace the vertices of degree two with $K_{3,3}$.
Suppose that~$e_1$ and~$e_2$ are incident edges to the vertex~$v$ with~$deg(v)=2$. 
Without loss of generality, we can assume that the orientation of~$e_1$ is toward~$v$.
It is not hard to see that there are~$a,b\in \mathbb Z_k\setminus\{0\}$ such that~$f(e_1)+a+b = 0$.
Consider the complete bipartite graph~$K_{3,3}$. 
Since the degree of each vertex is~$3$, we can find a non-elusive $\mathbb Z_k$-flow on~$K_{3,3}$ such that the value of all edges belong to the set~$\{f(e_1), a,b\}$. 
Suppose that~$e=v_1v_2$ of~$K_{3,3}$ with the value~$f(e_1)$ and the orientation from~$v_1$ to~$v_2$. 
We remove the edge~$e=v_1v_2$ from~$K_{3,3}$  and the vertex~$v$ of~$G$. Now, we join the edge~$e_1$ to~$v_2$ and~$e_2$ to~$v_1$.  
We repeat this process for all vertices of degree~$2$. 
Hence we obtain a cubic graph~$H'$  with a non-elusive $\mathbb Z_k$-flow and so~$\mathcal{F}(H')=\mathcal{F}'(H')=1$. Therefore~$\mathcal{F}'(G)=1$, as~$G$ is contraction of~$H'$. 
Hence we have shown that if~$\mathcal{F}(G)=1$, then~$\mathcal{F'}(G)=1$.

\noindent Now, if~$\mathcal{F}'(G)=1$, then the condition (iii) gives us an $H$ with $\mathcal{F}'(H)=1$, which $G$ is a contraction of $H$
and so~$\mathcal{F}(H)=1$. 
Thus we deduce that~$\mathcal{F}(G)=1$, as desired.
}\end{proof}
\noindent The proof of the preceding theorem implies the following corollary.
We note that as we mentioned before  ``contraction'' used in this paper is different from 
``minor'', see the footnote.
\begin{cor}
Every graph admitting  a  $k$-flow is a contraction of a cubic graph which has a $k$-flow.
\end{cor}
\noindent We now are ready to state the definition of semi-edge-colorability which was 
mentioned above.

\noindent {\bf{Definition 4:}} Let~$k$ be a positive integer. A \textit{semi-k-edge-coloring}  of a graph~$G$ is a map from~$E(G)$ to~$\{1,2,\ldots,k\}$, with the property
 that for every finite cut~$C$ of~$G$, if the number of edges of~$C$ with the color~$i$ is~$c_i$, then  the all numbers~$c_1,\ldots,c_k$ have  the same parity. A graph~$G$ is \textit{semi-k-edge-colorable}
  if~$G$ has a semi-$k$-edge-coloring.
  
\noindent  We use flows to characterize semi-edge-colorings. First, let~$V=\oplus_{i=1}^{k-1}  \mathbb Z_2$ be the vector space over~$\mathbb Z_2$ and~$e_i$ for~$i=1,\ldots,k-1$  be the standard basis. 
Set~$\mathcal{A}=\{e_1,\ldots,e_{k-1},\sum_{i=1}^{k-1}e_i\}$. 
Note that~$\mathcal{A}$ is compact with the discrete topology. 
We now use the notation of  \cite{thom2} and we state the following lemma. 
\begin{lem}
Let~$G$ be a finite graph and~$k$ be a positive integer. Then with the above notation, the following statements are equivalent.
\begin{enumerate}
 \item[{\rm (i)}]~$G$ is semi-$k$-edge-colorable.
\item[{\rm (ii)}]~$G$ has an~$\mathcal{A}$-flow.
\end{enumerate}
\end{lem}
\begin{proof}{
The one to one correspondence between the color set~$\{c_1,\ldots,c_k\}$ and~$\{e_1,\ldots,e_{k-1},\sum_{i=1}^{k-1}e_i\}$ induces a bijection between the set of semi-$k$-edge-colorings and the set of $\mathcal A$-flows.
}\end{proof}
 
 \noindent Immediately, Theorem \ref{Z_2} implies  the following 
  remark:
  \begin{remark}
Let~$G$ be a locally finite graph and~$k$ be a positive integer. Then the following statements are equivalent.
\begin{enumerate}
\item[{\rm (i)}]~$G$ is  semi-$2k$-edge-colorable.
\item[{\rm (ii)}]  The degrees of all vertices and ends of~$G$ are even.
\item[{\rm (iii)}]~$G$ has a non-elusive~$\mathbb Z_2$-flow.
\end{enumerate}

\end{remark}
 
\noindent  Our objective is to show that every 3-edge-colorable finite graph is a contraction of a cubic $3$-edge-colorable graph. 
In order to show this, we show that the  definition of semi-edge-coloring is the only definition which is compatible with the three properties of Theorem \ref{havingflow} for finite graphs, but instead of cubic graphs, we can have~$k$-regular graphs.
In other words, the Boolean functor having semi-$k$-edge-colorability is the unique Boolean functor which satisfies the conditions (i)-(iii) of Theorem \ref{havingflow}.\\

\begin{thm}\label{semi-k}
Let~$k$ be an odd number, let~$\mathcal{F}$ be the Boolean functor of  a finite graph being semi-$k$-edge-colorable and let~$\mathcal{F}'$ be another Boolean functor which satisfy the three following properties
\begin{enumerate}
    
\item[{\rm (i)}]~$\mathcal{F}$ and~$\mathcal{F}'$ are the same for~$k$-regular graphs.
\item[{\rm (ii)}] If~$\mathcal{F}'(G)=1$, then~$\mathcal{F}'(H)=1$ for every contraction~$H$ of~$G$.
\item[{\rm (iii)}] If~$\mathcal{F}'(G)=1$, then there is a finite~$k$-regular~$H$ such that~$G$ is a contraction of~$H$ with~$\mathcal{F}'(H)=1$.
\end{enumerate}  
Then~$\mathcal{F}$ and~$\mathcal{F}'$ are equal.
\end{thm}
\begin{proof}{
Assume that a graph $G$ is semi-$k$-edge-colorable and so~$\mathcal{F}(G)=1$. 
We construct a~$k$-regular graph~$H$ such that~$\mathcal{F}'(H)=1$ and moreover~$G$ is a contraction of~$H$. 
We notice that as we mentioned before the contracted vertex sets need not be connected
Let~$v$ be an arbitrary vertex of~$G$. 
If~$deg(v)=2n$, then  each color appears an even number of times, as the number of colors is odd and the degree is even.
Thus we are able to form pairs $P_i=\{e_i^1,e_i^2 \}$ of edges with the same color. 
Consider a~$k$-edge-coloring of the complete graph~$K_{k+1}$. 
We delete an edge~$e$ of the color of the edges of~$P_i$, join the edges in~$P_i$ 
to the end vertices of~$e$ in~$K_{k+1}$ and we denote by~$L$ the union of~$K_{k+1}\setminus\{e\}$ with edges~$e^1_{i}$ and~$e^2_{i}$. 
We do this for every~$P_i$ for~$i=1,\ldots,n$. 
If~$deg(v)=2n+1$, then  each color appears an odd number of times.
From each color, we choose an incident edge of $v$.
We separate them and we attach them to a new vertex $u$.
We notice that the degree of $u$ is $k$.
Thus  the number of colors appears in the rest of incident edges of $v$ is even.
Again we are able to pair these edges.
We do same for the paired edges as above.
Hence the vertex~$v$ is replaced by the union of some copies of~$L$ and the vertex $u$. 
Now, we do the same for every vertex of~$G$. 
Finally, we obtain a~$k$-edge-colorable~$k$-regular graph~$H$ which contains~$G$ as a contraction. 
Hence since~$\mathcal{F}'(H)=\mathcal{F}(H)=\mathcal{F}(G)=1$, we can conclude that~$\mathcal{F}'(G)=1$.\\
\noindent If~$\mathcal{F}'(G)=1$, then we note that semi-edge-colorability is preserved by contraction. 
So the first and third conditions imply that~$\mathcal{F}(G)=1$.
}\end{proof}
The proof of the preceding theorem implies the following corollary.
\begin{cor}\label{3-3}
Every $3$-edge-colorable finite graph is a contraction of a $k$-regular $3$-edge-colorable graph, where $k\geq3$ is an odd number.
\end{cor}
\noindent In finite cubic graphs, the existence of non-elusive~$\mathbb Z_4$-flows and~$3$-edge-colorability are 
equivalent, see \cite[Proposition 6.4.5]{diestel}. Next, we generalize this fact to infinite graphs.
\begin{thm}
  Let~$G$ be a graph. Then~$G$ has a non-elusive~$\mathbb Z_4$-flow if and only if~$G$ is 
  semi-$3$-edge-colorable.
\end{thm}
\begin{proof}{
First, assume that~$G$ is semi-$3$-edge-colorable. Since every contraction of~$G$ is semi-$3$-edge-colorable,
we conclude that every~$G_M$ is semi-$3$-edge-colorable,  for every finite subset of~$M$ of~$\mathcal{B}_{\rm{fin}}(G)$. 
It follows from Corollary \ref{3-3} that there is a cubic graph~$\widetilde{G_M}$ in such a way that~$\widetilde{G_M}$ is~$3$-edge-colorable and moreover~$G_M$ is a contraction of~$\widetilde{G_M}$. We invoke Part (ii) of \cite[Proposition 6.4.5]{diestel} and we conclude that~$\widetilde{G_M}$ has a non-elusive~$\mathbb Z_4$-flow, as~$\widetilde{G_M}$ is a cubic graph and it is~$3$-edge-colorable. 
We notice  that by the definition of~$\widetilde{G_M}$, we deduce that~$G_M$ has a non-elusive~$\mathbb Z_4$-flow. 
Now, by Theorem~\ref{compact}, we deduce that~$G$ has a non-elusive~$\mathbb Z_4$-flow. 
For the forward implication, 
by Corollary \ref{HH'},~$G$ has a non-elusive~$\mathbb Z_2\oplus\mathbb Z_2$-flow, say~$f$. 
We define a semi-$3$-edge-coloring~$c\colon E(G)\rightarrow \mathbb Z_2\oplus\mathbb Z_2\setminus\{(0,0)\}$ 
by~$c(e)=f(e)$. 
Let~$F$ be a finite cut of~$G$. 
Then since $f$ is a non-elusive $\mathbb Z_2\oplus \mathbb Z_2$, the map $f$ sums up to zero on the edges of $F$.
In particular, the sum of all assignments of edges with the value $(1,0)$ is zero.
Thus we are able to deduce that the parity of every color of each edge of $F$ is the same. 
Thus~$G$ is semi-$3$-edge-colorable, as desired.~}\end{proof}

 

\subsection{Hamiltonicity}

\noindent A  graph is \textit{Eulerian} if it is connected and all vertices have even degree. We call a finite graph \textit{supereulerian} if it has a spanning Eulerian subgraph.

\begin{lem}\label{supereulerian}
 Every finite supereulerian graph has a non-elusive~$\mathbb Z_4$-flow.
\end{lem}
\begin{proof}{
Let~$G$ be a supereulerian graph. 
Then by Corollary \ref{HH'}, it is enough to show that~$G$ has a non-elusive~$\mathbb Z_2\oplus\mathbb Z_2$-flow. 
Let~$C$ be a spanning Eulerian subgraph of~$G$.  The degree of every vertex of~$G$ in~$C$ is even. 
Thus the constant function with the value $(0,1)$ is a non-elusive $\mathbb Z_2$-flow in $C$. 
We denote this~$\mathbb Z_2\oplus\mathbb Z_2$-flow by~$F$. 
Let~$e_1,\ldots,e_k$ be an enumeration of the edges outside~$C$. 
Suppose that~$u_i$ and~$v_i$ are the end vertices of~$e_i$. 
Since~$C$ is a spanning Eulerian subgraph of~$G$, we can find a walk~$P_i$ in $C$ between~$u_i$ and~$v_i$.
We define a new flow~$F_i$ by assigning~$(1,0)$ to every edge of~$P_i\cup\{e_i\}$. 
Note that~$P_i\cup\{e_i\}$ is an Eulerian subgraph. 
So~$F_i$ is a~$\mathbb Z_2\oplus\mathbb Z_2$-flow of~$G$, for~$i=\{1,\ldots,k\}$.
Then~$\sum_{i=1}^k F_i+F$ is a~$\mathbb Z_2\oplus\mathbb Z_2$-flow, too. 
Now, we claim that~$\sum_{i=1}^k F_i+F$ is a non-elusive~$\mathbb Z_2\oplus\mathbb Z_2$-flow. 
It is enough to show that~$\sum_{i=1}^k F_i+F$ is non-zero for an arbitrary edge of~$C$, as the value of~$e_i$ is~$(1,0)$, for~$i=1,\ldots,k$. 
Since the second component of the map~$\sum_{i=1}^k F_i+F$ is always 1 for every edge of~$C$, the flow~$\sum_{i=1}^k F_i+F$ is non-elusive and the claim is proved, as desired.
}\end{proof}
\begin{remark}\label{18}
Catlin \cite{Catlin} showed that every finite~$4$-edge-connected graph is supereulerian. 
Thus it follows from Lemma \ref{supereulerian} that every finite~$4$-edge-connected graph has a~$4$-flow. This result has been proved by Jaeger \cite{Jaeger1}.
\end{remark}
\noindent A \textit{Hamiltonian circle} is a circle containing every vertex of an infinite graph. 
It is worth mentioning that every
Hamiltonian circle contains all vertices and all ends precisely once.
\begin{cor}
Every graph containing a Hamiltonian circle has a non-elusive~$\mathbb Z_4$-flow.
\end{cor}
\begin{proof}{
Let~$C$ be a Hamiltonian circle of~$|G|$ and~$M$ be a finite subset of~$\mathcal{B}_{\rm{fin}}(G)$.
Also, let~$\phi:G\to G_M$ be the map which is defined in Remark \ref{phi}. 
Then~$\phi(C)$ is a spanning Eulerian subgraph of~$G_M$ and so~$G_M$ is supereulerian.
It follows from Lemma \ref{supereulerian} that~$G_M$ has a non-elusive~$\mathbb Z_4$-flow 
for every finite subset~$M$ of~$\mathcal{B}_{\rm{fin}}(G)$. 
Now, we invoke Theorem \ref{compact} and we conclude that~$G$ has a non-elusive~$\mathbb Z_4$-flow.}\end{proof}

\subsection{Conjectures}

\noindent In the study of flow theory one main point of interest is the connection to the edge-connectivity. For example, if a finite graph is~$2$-edge-connected, then it has a non-elusive~$\mathbb Z_6$-flow, see \cite{sey}.\\
Next up, we show that the connection between edge-connectivity  and the existence of  a non-elusive flow for infinite graphs admits exactly the same connection as for finite graphs.
\begin{cor}\label{edgeconn}
If~$n$-edge-connectivity implies the existence of an~$m$-flow for finite graphs, then this  implication holds for infinite graphs as well.
\end{cor}

\begin{proof}{ Let~$M$ be a finite subset of~$\mathcal{B}_{\rm{fin}}(G)$.
Note that since~$G$ is~$n$-edge-connected, the multigraph~$G_M$ is~$n$-edge-connected. 
By assumption, the graph $G_M$ has an~$m$-flow.
Now, we invoke Theorem \ref{compact} and  conclude that~$G$ has an~$m$-flow.
}\end{proof}
As a corollary of Remark \ref{18} and Corollary \ref{edgeconn}, we obtain the following.
\begin{cor}
Every~$4$-edge-connected graph has a~$4$-flow.\hfill~$\square$ 
\end{cor}
\noindent  There are some famous conjectures in finite flow theory such as the four-flow conjecture and the three-flow conjecture. If these conjectures hold true for finite graphs, then they are true for infinite graphs and vice versa.

\noindent {\bf Five-flow conjecture:}  Every~$2$-edge-connected graph  has a~$5$-flow.

\noindent {\bf Four-flow conjecture:} Let~$G$ be a bridgeless graph. 
If for every finite subset $M$ of~$\mathcal{B}_{\rm{fin}}(G)$, $G_M$ does not contain the  Petersen graph as a topological minor, then~$G$ has a non-elusive~$4$-flow.

\noindent {\bf Three-flow conjecture:} Every~$4$-edge-connected graph has a~$3$-flow.

\noindent In 1961, Seymour \cite{sey}  has shown that every finite bridgeless graph has 
a~$6$-flow. 
Immediately, Theorem \ref{compact} implies the following theorem.
\begin{thm}
Every  bridgeless graph~$G$ has a~$6$-flow.\hfill~$\square$ 
\end{thm}

\section{Tension of Infinite Graphs}

Another concept related to flows is tension.  Let~$G=(V,E)$ be a finite graph and~$K$ be a group that is not necessarily  abelian. 
We call a map~$f:\overrightarrow{E}\to K$ a \textit{$K$-tension} if~$f$ satisfies~$\sum_{e\in \overrightarrow{C}} f(e)=0$ for every directed cycle~$\overrightarrow{C}$ of~$G$. 
We note that we sum up the assignments of edges with respect to a cyclic order.
If~$f(\overrightarrow{e})\neq 0$ for every~$\overrightarrow{e}\in \overrightarrow{E}$ then~$G$ has a \textit{nowhere-zero~$K$-tension}. Since we are studying cycles, it does not matter where we start, and moreover, if~$G$ has a~$K$-tension, the choice of our edge orientation is irrelevant, as every element of~$K$ has its inverse. So we can define our~$K$-tension for infinite graphs~$G$ in an analogous manner with superseding finite cuts with finite cycles in the definition of a flow. 
Suppose that~$K$ is a Hausdorff topological group with a compact subset~$A$ of~$K$. 
We define~$\sigma \colon K^{\overrightarrow{E}}\to K^{\overrightarrow{\mathcal{C}_{\rm{fin}}}(G)}$ such that~$\sigma(f)(\overrightarrow{C})=\sum_{\overrightarrow{e}\in \overrightarrow{C}} f(\overrightarrow{e})$ for any finite oriented cycle~$C$. Let~$M$ be a subset of~${\mathcal{C}_{\rm{fin}}}(G)$.
Then we say that~$G$ has an \textit{$A$-tension} with respect to~$M$  if~$F_M=\{f\in A^{\overrightarrow{E}}\,|\,\sigma(f)(C)=0$ for every~$C\in M\}$ is not empty and we say that~$G$ has an \textit{$A$-tension} if~$G$ has an~$A$-tension with respect to~$\mathcal{C}_{\rm{fin}}(G)$. 
If~$f$ is an~$A$-tension and $K\setminus \{0\}$ then we say that~$f$ is a \textit{non-elusive~$K$-tension}. 

If $f$ is an $A$-flow and $A\subseteq K\setminus \{0\}$, then we also call $f$ a \emph{non-elusive $H$-flow}. 
Now, a natural question arises: When does an infinite graph~$G$ have a non-elusive~$K$-tension? 
At first glance, it seems that we can use the concept of dual graphs. 
A pair of \textit{dual graphs} is a pair of graphs~$(G, G^*)$ such that there is a bijection~$\phi: E(G)\to E(G^*)$ with the property that a finite set~$A\subseteq E(G)$ is the edge set of a cycle if and only if~$\phi(A)$ is a bond (minimal edge cut)  in~$G^*$. Thomassen \cite[Theorem 3.2]{thom1} showed that  a~$2$-connected graph~$G$ has a dual graph if and only if~$G$ is planar and any two vertices of~$G$ are separated by a finite edge cut. 
Moreover if~$G^*$ is a dual graph of~$G$ and~$A\subseteq E(G)$, then~$G^*/A^*$ is a dual graph of~$G-A$, see \cite[Lemma 9.11]{thom3}. 
For more details regarding the concept of duality with the topological approach, see \cite{Bruhn}. We denoted by~$G^*/A^*$ the graph obtained from~$G^*$ by contracting all edges of~$A^*$.  
Hence, for defining the similar graph like~$G_M$ in Definition 3, we have to delete some edges from~$G$ and it holds true only for planar graphs where every two of its vertices are separated by a finite edge cut. 
In the next theorem, we delete edges for an arbitrary graph and show that the above argument is still true.
\begin{thm}
Let~$G$ be a graph and~$\mathcal{C}$ be a finite subset of~$\overrightarrow{\mathcal{C}}_{\rm{fin}}(G)$.  
Then~$G$ has a non-elusive~$K$-tension if and only if every  finite subset~$\mathcal{C}$ of~$\overrightarrow{\mathcal{C}}_{\rm{fin}}(G)$ has a non-elusive~$K$-tension.
\end{thm}
\begin{proof}
{ Set~$$F_{\mathcal{C}}=\{f\,{\rm{is \,a}}\,K\textendash{\rm{tension\,of\, \textit{G}}} \mid f\, \rm{\,is\, a\,nowhere}\textendash \rm{zero}\,\,\emph{K} \textendash{tension\,with\,respect\,to} \,\mathcal{C}\}.$$
Then~$F_{\mathcal{C}}$ is not empty for any finite subset~$\mathcal{C}$ of~$\overrightarrow{\mathcal{C}}_{\rm{fin}}(G)$. 
So using an analogous method as in the proof of Theorem \ref{compact},  we  conclude that~$G$ has a non-elusive~$K$-tension.
}\end{proof}
\noindent{\bf Acknowledgements.}  The authors are deeply grateful to the referees for careful reading.
Also the authors are grateful to Pascal Gollin, Matthias Hamann and Peter Christian Heinig for their comments. 



{}

\end{document}